\documentclass[11pt]{article}

\usepackage{yhmath}
\usepackage[english]{babel}
\usepackage{enumerate}
\usepackage{amsthm}
\usepackage{caption}
\usepackage{amsmath}
\usepackage{amssymb}
\usepackage{amsfonts}

\usepackage{mathdots}

\newtheorem{thm}{Theorem}[section]
\newtheorem*{thm*}{Theorem} 
\newtheorem{lem}[thm]{Lemma}
\newtheorem{Def}[thm]{Definition}

\newtheorem{prop}[thm]{Proposition}

\newtheorem{rem}[thm]{Remark}

\usepackage{eqlist} 

\newcommand{\N}{\ensuremath{\mathbb{N}}}

\newcommand{\G}{\ensuremath{\mathcal{G}}}
\newcommand{\Hcal}{\ensuremath{\mathcal{H}}}

\newcommand{\del}{\ensuremath{\partial}}

\newcommand{\C}{\ensuremath{\mathbb{C}}}

\newcommand{\Q}{\ensuremath{\mathbb{Q}}}

\newcommand{\Gal}{\ensuremath{\mathrm{Gal}}}

\newcommand{\GL}{\operatorname{GL}}

\newcommand{\Aut}{\operatorname{Aut}}

\newcommand{\zu}{\hspace{-0.1cm}>\hspace{-0.15cm}}

\title{Parameterized Differential Equations over $k((t))(x)$}
\author{Annette Maier \footnote{Lehrstuhl f\"ur Mathematik (Algebra), RWTH Aachen University,
   52062 Aachen, Germany. Email: annette.maier@matha.rwth-aachen.de}
}

\date{\today}

\begin{document}
\date{\today}
\maketitle
\begin{abstract} In this article, we consider the inverse Galois problem for parameterized differential equations over $k((t))(x)$ with $k$ any field of characteristic zero and use the method of patching over fields due to Harbater and Hartmann. As an application, we prove that every connected semisimple $k((t))$-split linear algebraic group is a parameterized Galois group over $k((t))(x)$.
\end{abstract}
\textit{2010 Mathematics Subject Classiﬁcation.} 12H05, 20G15, 14H25, 34M03, 34M15, 34M50. 


\section*{Introduction} 
In classical Galois theory, we are given a polynomial $f$ over a field $F$ and consider the field $E$ obtained by adjoining all roots of $f$ inside an algebraic closure. The Galois group is the finite group of all automorphisms of $E$ that act trivially on $F$. The inverse problem asks which finite groups are Galois groups over a given field $F$. For example over $\Q$ this is still an open problem. In differential Galois theory, we start with a linear differential equation over a differential field $F$ and look at the automorphisms of the field obtained by adjoining a complete set of solutions that act trivially on $F$ and commute with the derivation. This group measures the algebraic relations among the solutions. Parameterized differential Galois theory is a refinement of differential Galois theory where the Galois group measures algebraic relations as well as the $\del_t$-algebraic relations among the solutions, if the base field $F$ is equipped with an additional derivation $\del_t$ depending on a parameter $t$. \\
\\
Let $k$ be a field of characteristic zero and define $F=k((t))(x)$ with the two natural derivations $\del_x$ and $\del_t$. Consider a linear differential equation $\del_x(y)=Ay$ for some $A \in F^{n\times n}$. The Galois theory of parameterized differential equations assigns a parameterized Galois group to this equation which can be identified with a linear differential group $\G \leq \GL_n$ defined over $k((t))$. That means that $\G\leq \GL_n$ is given by $\del_t$-differential polynomials in the coordinates of $\GL_n$ over $k((t))$. The inverse Galois problem in this situation asks which linear differential groups defined over $k((t))$ can be obtained as parameterized Galois groups of some differential equations. \\
\\
So far, the inverse Galois problem for parameterized differential equations has only been considered over $U(x)$ where $U$ is equipped with a derivation $\del_t$ (or more generally with several commuting derivations $\del_{t_1},\dots,\del_{t_r}$) and is differentially closed or even stronger, a universal differential field. This means that for any differential field $L\subseteq U$ which is finitely differentially generated over $\Q$, any differentially finitely generated extension of $L$ can be embedded into $U$. Over such a field $U(x)$, the following necessary (\cite{Dreyfus}) and sufficient (\cite{MR2975151}) condition was recently found: A linear differential group $\G$ defined over $U$ is a parameterized Galois group if and only if $\G$ is differentially finitely generated. That is, there are finitely many elements $g_1,\dots, g_m \in \G(U)$ such that $\G(U)$ is the smallest differentially closed subgroup of $\GL_n(U)$ containing them. In the special case of $\G$ a linear algebraic group over $U$, Singer then showed that $\G$ is differentially finitely generated if and only if the identity component $\G^\circ$ has no quotient isomorphic to the additive group $\mathbb{G}_a$ or the multiplicative group $\mathbb{G}_m$ (\cite{MR2995020}). This implies in particular that every semisimple linear algebraic group defined over $U$ is a parameterized Galois group over $U(x)$. \\
\\
Over fields $U(x)$ with $U$ not differentially closed, not much is known on the inverse problem. We restrict ourselves to the base field $F=k((t))(x)$ as above. This is the function field of a curve over a complete discretely valued field. Over such a field we can apply the method of patching over fields due to Harbater and Hartmann (see Theorem \ref{patchingarithmcurves}). This method has been applied by Harbater and Hartmann to (non-parameterized) differential Galois theory (see \cite{HHOW}). We give an application of patching to parameterized Galois theory (Theorem \ref{thmpatching}) which states that in order to have the existence of a linear differential equation over $F$ with Galois group a given group $\mathcal{G}$, it is sufficient to construct $r$ linear differential equations over certain overfields $F_1,\dots,F_r$ with certain properties and Galois groups $\G_1,\dots,\G_r$ such that $\G_1,\dots,\G_r$ generate $\G$ as a linear differential group. This allows to break down the problem into smaller pieces. If $\G$ is a $k((t))$-split semisimple linear algebraic group, we show that $\G$ can be differentially generated by suitable subgroups $\G_1,\dots,\G_r$ of its root subgroups (see Proposition \ref{propgen}) by using a theorem of Cassidy that classifies Zarisiki-dense linear differential subgroups of $\G$ (see \cite[Thm. 19]{Cassidyclassification}). By realizing these groups $\G_1,\dots,\G_r$ as parameterized Galois groups over suitable overfields $F_1,\dots,F_r$, we then prove our main result (Theorem \ref{result}):
\begin{thm*}\nonumber
Let $\Hcal \leq \GL_n$ be a connected semisimple linear algebraic group defined and split over $k((t))$. Then $\Hcal$ is the parameterized Galois group of some $n$-dimensional $\del_x$-differential equation over $k((t))(x)$. 
\end{thm*} 
This paper is organized as follows. In Section \ref{1}, we provide some background on the Galois theory of parameterized differential equations. In Section \ref{secpat}, we present the method of patching as established by Harbater and Hartmann and give an application to parameterized Galois theory. Section \ref{3} deals with how to differentially generate a semisimple linear algebraic group. The main theorem is then proven in Section \ref{4}.\\ 
\\
\noindent \textbf{Acknowledgments.} I would like to thank Julia Hartmann and Michael Wibmer for helpful discussions during the preparation of this manuscript.

\section{Parameterized differential equations}\label{1}
All fields are assumed to be of characteristic zero and all rings are assumed to contain $\Q$.
A ${\del_t}\del$-ring $R$ is a ring $R$ with two commuting derivations $\del$ and $\del_t$. Examples of such rings are $\C[t][x]$, $\C[[t]][x]$, $\C(t)(x)$, $\C((t))(x)$. A ${\del_t}\del$-field is a ${\del_t}\del$-ring that is a field. Homomorphisms of ${\del_t}\del$-rings are homomorphisms commuting with the derivations, ${\del_t}\del$-ideals are ideals stable under the derivations and ${\del_t}\del$-ring extensions are ring extensions with compatible ${\del_t}\del$-structures. The $\del$-constants $C_R$ are the elements of a ${\del_t}\del$-ring $R$ mapped to zero by $\del$. A linear $\del$-equation $\del(y)=Ay$ with a matrix $A \in F^{n\times n}$ over a ${\del_t}\del$-field $F$ is also called a parameterized (linear) differential equation to emphasize the extra structure ${\del_t}$ on $F$. A fundamental solution matrix for $A$ is an element $Y \in \GL_n(R)$ for some ${\del_t}\del$-ring $R/F$ such that $\del(Y)=AY$ holds. 
\begin{Def}
Let $\del(y)=Ay$ be a parameterized differential equation over a ${\del_t}\del$-field $F$. A \emph{parameterized Picard-Vessiot extension} for $A$, or PPV-extension for short, is a ${\del_t}\del$-field extension $E$ of $F$ such that 
\begin{enumerate}
 \item There exists a fundamental solution matrix $Y \in \GL_n(E)$ such that ${E=F\<Y\zu_{\del_t}}$ which means that $E$ is generated as a field over $F$ by the coordinates of $Y$ and all its higher derivatives with respect to ${\del_t}$. 
\item $C_E=C_F$.
\end{enumerate}
A \emph{parameterized Picard-Vessiot ring} for $A$, or PPV-ring for short, is a ${\del_t}\del$-ring extension $R/F$ such that 
\begin{enumerate}
\item There exists a fundamental solution matrix $Y \in \GL_n(R)$ such that ${R=F\{Y,Y^{-1}\}_{\del_t}}$, i.e., $R$ is generated as an $F$-algebra by the coordinates of $Y$ and $\det(Y)^{-1}$ and all their higher $\del_t$-derivatives.
\item $C_R=C_F$.
\item $R$ is $\del$-simple, i.e. $R$ has no nontrivial $\del$-invariant ideals. 
\end{enumerate}
\end{Def}
A Picard-Vessiot ring for $A$ always exists if $F^\del$ is algebraically closed, see \cite{Wibparam}. Every PPV-extension contains a unique PPV-ring. Indeed, let $E$ be a PPV-extension with fundamental solution matrix $Y \in \GL_n(E)$. Then $R:=F\{Y,Y^{-1}\}_{\del_t}$ is a PPV-ring for $M$. (It is not obvious that $R$ is $\del$-simple, though. This can be shown by writing $R$ as an infinite union of non-parameterized Picard-Vessiot rings corresponding to prolongations of $A$ and then applying the corresponding statement from the non-parameterized Picard-Vessiot theory.) Uniqueness follows from the fact that if $Y' \in \GL_n(E)$ is a fundamental solution matrix for $M$, then there exists a $B \in \GL_n(C_E)\subseteq \GL_n(F)$ with $Y'=YB$, hence ${F\{Y',Y'^{-1}\}_{\del_t}=F\{Y,Y^{-1}\}_{\del_t}}$. 
\begin{Def}
Let $\del(y)=Ay$ be a parameterized differential equation over a ${\del_t}\del$-field $F$ such that there exists a PPV-ring $R$ for $A$. Denote $C=C_F$. Then the \emph{Galois group of $A$} (with respect to $R$) is the group functor 
\[ \Gal(A) \colon \underline{{\del_t}\hspace{-0.14cm}-\hspace{-0.11cm}C\hspace{-0.14cm}-\hspace{-0.14cm}\text{algebras}}\to \underline{\text{Groups}}, \ S \mapsto \operatorname{Aut}^{{\del_t}\del}(R\otimes_C S / F \otimes _C S), \] where $\operatorname{Aut}^{{\del_t}\del}(R\otimes_C S / F \otimes _C S)$ denotes the ${\del_t}\del$-compatible automorphisms of $R\otimes_C S$ (which is considered as a ${\del_t}\del$-ring via $\del|_S=0$) that act trivially on $F\otimes_C S$.
\end{Def}

Let $(C,\del_t)$ be a differential field. A \emph{linear $\del_t$-group over $C$} is a group functor $\G \colon \underline{{\del_t}\hspace{-0.14cm}-\hspace{-0.11cm}C\hspace{-0.14cm}-\hspace{-0.14cm}\text{algebras}}\to \underline{\text{Groups}}$ such that there exists an $n \in \N$ and a set $\{p_1,\dots,p_m\}\subseteq C[\del_t^{k}(X_{ij}) \ | \ k \in \N_{\geq 0}, 1 \leq i,j \leq n]$ of $\del_t$-polynomials in $n^2$ variables with coefficients in $C$ such that for all $\del_t$-algebras $S$ over $C$: ${\G(S)=\{g \in \GL_n(S) \ | \ p_i(g)=0, \ i=1,\dots, m \}}$.

\begin{thm}\label{thmgalois}
Let $F$ be a $\del_t\del$-field and let $A \in F^{n\times n}$. Assume that there exists a PPV-ring $R$ for the parameterized differential equation $\del(y)=Ay$. Denote $C=C_F$. Then the following holds.
\begin{enumerate}[a)]
\item \label{galois1} $\Gal(A)$ becomes a linear ${\del_t}$-group over $C$ via the following natural embedding $\Gal(A) \leq \GL_n$ depending on a fixed fundamental solution matrix $Y \in \GL_n(R)$:
\[\theta_S\colon\operatorname{Aut}^{{\del_t}\del}(R\otimes_C S / F \otimes _C S)\hookrightarrow \GL_n(S), \ \sigma \mapsto (Y\otimes 1)^{-1}\sigma(Y\otimes 1).\] 
Let $J$ be the kernel of the ${\del_t}$-$F$-homomorphism $F\{X,X^{-1}\}_{\del_t} \to R$ given by $X \mapsto Y$, where $X$ consists of $n^2$ ${\del_t}$-differentially-independent variables. Then the image of $\theta_S$ equals $$\{g \in \GL_n(S) \ | \ f(Yg)=0 \text{ for all } f \in J\}.$$  
\item\label{galois3} Let $\frac{a}{b}$ be an element of the field of fractions of $R$ (note that $R$ is a domain since it is $\del$-simple). If $\frac{a}{b}$ is functorially invariant under the action of $\Gal(A)$, i.e., for every ${\del_t}$-$C$-algebra $S$ and every \\ $\sigma \in \Aut^{{\del_t}\del}(R\otimes_C S/ F\otimes_C S)$ we have 
\[\sigma(a\otimes_C 1)\cdot(b\otimes_C 1)=(a \otimes_C 1)\cdot \sigma(b\otimes_C 1), \] then $\frac{a}{b}$ is contained in $F$. 
\end{enumerate}
\end{thm}
\begin{proof}
 It is easy to see that $\theta_S$ is a well-defined and injective group homomorphism for every $S$. To see that $\operatorname{Im}(\theta_S)=\{g \in \GL_n(S) \ | \ f(Yg)=0 \text{ for all } f \in J\}$, note that $R$ is isomorphic to $F\{X,X^{-1}\}_{\del_t}/J$ as a ${\del_t}$-ring. We extend $\del$ to $F\{X,X^{-1}\}_{\del_t}$ via $\del(X)=AX$. Then $J$ is a ${\del_t}\del$-ideal and $R\cong~ F\{X,X^{-1}\}_{\del_t}/J$ as ${\del_t}\del$-rings. Let $J_S \subseteq (F\otimes_C S)\{X,X^{-1}\}_{\del_t}$ be the ideal generated by $J$. Note that $J_S$ is a ${\del_t}$-ideal, since $J$ and $S$ are closed under $S$. Then $$R\otimes_C S \cong (F\{X,X^{-1}\}_{\del_t}/J)\otimes_C S\cong (F\otimes_C S)\{X,X^{-1}\}_{\del_t}/J_S.$$ Now
$X\mapsto X\cdot g$ defines a ${\del_t}\del$-$(F\otimes_C S)$-automorphism of $(F\otimes_C S)\{X,X^{-1}\}_{\del_t}$ for every $g \in \GL_n(S)$.  Hence $g \in \GL_n(S)$ induces a ${\del_t}\del$-automorphism on $R\otimes_CS$ if and only if it leaves $J_S=\{f\in S\{X,X^{-1}\}_{\del_t} \ | \ f(Y)=0\}$ invariant. We conclude $\theta_S(\Gal(A)(S))=\{g \in \GL_n(S) \ | \ f(Yg)=0 \text{ for all } f \in J_S\}=\{g \in \GL_n(S) \ | \ f(Yg)=0 \text{ for all } f \in J\},$ since $J$ generates $J_S$. In particular, $\Gal(A)$ is a linear $\del_t$-group defined over $F\<Y\zu_{\del_t}$. See \cite[Lemma 8.2]{Gilletetc} for a proof that it is in fact defined over $C$.\\
The second statement follows from the Galois correspondence of parameterized differential modules (\cite[Proposition 8.5]{Gilletetc}).
\end{proof}
Whenever we write $\Gal(A)\leq \GL_n$, this is understood to be with respect to a fundamental solution matrix that was fixed beforehand. (Another choice of a fundamental solution matrix inside the same Picard-Vessiot ring would yield a conjugate image inside $\GL_n$).

\section{Patching parameterized differential equations}\label{secpat}
Patching over fields is a method which was established by Harbater and Hartmann in \cite{HH} and which has been applied to differential modules (see \cite{HHOW}). We give a related application of patching to parameterized differential modules in Theorem \ref{thmpatching}. The method of patching can be applied over fields of transcendence degree one over complete discretely valued fields. We restrict ourselves to the situation $F=k((t))(x)$ for a field $k$ of $\operatorname{char}(k)=0$. We fix pairwise distinct elements $q_1,\dots,q_r \in k$ and define fields $F_0$ and $F_i, F_i^\circ$ for $1 \leq i \leq r$ as follows:\\
\\
{\underline{Setup:} \vspace{-0.8cm}
\begin{eqnarray*}
F&=& k((t))(x) \\
F_0&=&\operatorname{Frac}(k[(x-q_1)^{-1},\dots,(x-q_r)^{-1}][[t]])\\
F_i&=&\operatorname{Frac}(k[[t]][[x-q_i]]) \\
F_i^\circ&=&k((x-q_i))((t)).
\end{eqnarray*}}
Note that $k[[t]][[x-q_i]]=k[[x-q_i]][[t]]$, hence $F \subseteq F_i \subseteq F_i^\circ$ for each $1 \leq i \leq r$. Also, $F\subseteq F_0$ and $F_0\subseteq k(x)((t))\subseteq F_i^\circ$ for each $i$, hence we have a diagram of fields $F \subseteq F_0, F_i \subseteq F_i^\circ$ for each $1\leq i \leq r$. 

\begin{thm}[Harbater-Hartmann]\label{patchingarithmcurves} $\text{}$ \\ \vspace{-0.6cm}
\begin{enumerate}
 \item Let $x \in F_0$ be such that for each $1\leq i \leq r$, $x$ is contained in $F_i$ (when considered as an element inside $F_i^\circ$). Then $x \in F$. 
\item Let $n \in \N$ and $Y_i \in \GL_n(F_i^\circ)$ for $1\leq i \leq r$. Then these matrices can be simultaneously factored as follows: There exist matrices $Z_i \in \GL_n(F_i)$ for $1\leq i \leq r$ and one matrix $Y \in \GL_n(F_0)$ such that $Y_i=Z_i^{-1}Y$ holds for each $1 \leq i \leq r$. 
\end{enumerate}
\end{thm}
\begin{proof}
We explain how this can be obtained from Theorem 5.10 in \cite{HH}. Set $T=k[[t]]$, $\hat X=\mathbb{P}^1_T$ and let $S$ be the set consisting of the points $Q_1,\dots,Q_r$ on $X=\mathbb{P}^1_k$ given by $q_1,\dots,q_r \in k$. Then the local ring of $\hat X$ at $Q_i$ is $k[[t]][x]_{(t,x-q_i)}$ with $(t,x-q_i)$-adic completion $\hat R_{i}=k[[t]][[x-q_i]]$. Hence $\operatorname{Frac}(\hat{R_i})=F_i$ for $1\leq i \leq r$. The $t$-adic completion $\hat R_i^\circ$ of the localization of $\hat{R_i}$ at the height one prime $t\hat{R_i}$ equals $k((x-q_i))[[t]]$ with fraction field $k((x-q_i))((t))=F_i^\circ$. Hence in this special setup, the fields $F_i$, $F_i^\circ$ as defined above coincide with the fields $F_i=\operatorname{Frac}(\hat R_i)$, $F_i^\circ=\operatorname{Frac}(\hat R_i^\circ)$ as defined in \cite{HH}.\\ Define further $U=X$ and $U'=X\backslash S$. Then the ring $R_{U'}\subseteq F$ of rational functions regular on $U'$ equals $\mathcal{S}^{-1}(k[[t]][(x-q_1)^{-1},\dots,(x-q_r)^{-1}])$, where $\mathcal{S}$ denotes all elements that are units modulo $t$. Its $t$-adic completion $\hat R_{U'}$ equals ${k[(x-q_1)^{-1},\dots,(x-q_r)^{-1}][[t]]}$ (compare the argument in \cite{HH} in the example following Theorem 5.9 on page 85-86). Hence with the notation as in \cite{HH}, $F_{U'}=F_0$ and $F_U=F$. \\
Therefore, \cite[Thm. 5.10]{HH} implies that there is an equivalence of categories 
\[\operatorname{Vect}(F) \rightarrow \prod \limits_{i=1}^r \operatorname{Vect}(F_i) \times_{\prod \limits_{i=1}^r \operatorname{Vect}(F_i^\circ)} \operatorname{Vect}(F_0), \] which implies Part (b) of this theorem (see for example \cite[Prop. 2.2]{HHKtorsor}). Part (a) follows from \cite[Prop. 6.3]{HH} with again $S=\{Q_1,\dots,Q_r\}$ and now $U=X\backslash S$.
\end{proof}

Let now $\del=\frac{\del}{\del x}$ and $\del_t=\frac{\del}{\del t}$ be the natural derivations on $F$. Note that $\del$ and $\del_t$ extend to all fields $F_i$, $F_i^\circ$ and $F_0$ compatibly with the inclusions $F \subseteq F_i, F_0 \subseteq F_i^\circ$. Also note that $C_F=k((t))=C_{F_i^\circ}$ for all $1 \leq i \leq r$ and in particular $C_{F_i}=k((t))=C_F$ for $i=0,\dots,r$. 

\begin{thm}\label{thmpatching} Let $n \in \N$.
For $1 \leq i \leq r$, let $A_i \in F_i^{n\times n}$ be such that there exists a fundamental solution matrix $Y_i \in \GL_n(F_i^\circ)$ for the parameterized $\del$-equation $\del(y)=A_iy$ over $F_i$. Let $\G_i \leq \GL_n$ be the Galois group of $A_i$. Then there exists a parameterized $\del$-equation $\del(y)=Ay$ over $F$ with fundamental solution matrix $Y \in \GL_n(F_0)$ and corresponding Galois group $\G\leq \GL_n$ over $F$ satisfying $\G_i\leq \G$ for each $1\leq i \leq r$. Furthermore, $\G$ is the smallest $\del_t$-closed-subgroup of $\GL_n$ that contains $\G_i$ for all $1\leq i \leq r$. In other words, $\G$ is the Kolchin closure of the group $<\G_1,\dots,\G_r>$ generated by all $\G_i$.
\end{thm}
\begin{proof} We abbreviate $C=C_F=k((t))$. We apply Part b) of Theorem \ref{patchingarithmcurves} and obtain matrices $Z_i \in \GL_n(F_i)$ for $1\leq i \leq r$ and $Y \in \GL_n(F_0)$ such that $Y_i=Z_i^{-1}Y$ holds for each $1 \leq i \leq r$. Consider $A:=\del(Y)Y^{-1} \in F_0^{n\times n}$. For each $1 \leq i \leq r$, we can consider $A$ as an element in $(F_i^\circ)^{n\times n}$ and compute $A=\del(Y)Y^{-1}=\del(Z_iY_i)\cdot (Z_iY_i)^{-1}=\del(Z_i)Z_i^{-1}+Z_iA_iZ_i^{-1} \in F_i^{n\times n}$. A coefficient-wise application of Part a) of Theorem \ref{patchingarithmcurves} now implies that $A$ is contained in $F^{n\times n}$. Let $\G \leq \GL_n$ denote the Galois group of $\del(y)=Ay$ over $F$ corresponding to the fundamental solution matrix $Y \in \GL_n(F_0)$.  \\
\\
Let $S$ be a ${\del_t}$-algebra over $C$. Then by Theorem \ref{thmgalois}\,\ref{galois1}), $$\G(S)=\{g \in \GL_n(S) \ | \ f(Yg)=0 \text{ for all }f \in J \ \}$$ with $J=\{ f \in F\{X,X^{-1}\}_{\del_t} \ | \ f(Y)=0\}$ and
$$\G_i(S)=\{g \in \GL_n(S) \ | \ f(Y_ig)=0 \text{ for all } f\in J_i \ \}$$ with $J_i=\{ f \in F_i\{X,X^{-1}\}_{\del_t} \ | \ f(Y_i)=0\}\supseteq \{f(Z_iX) \ | \ f \in J \}.$\\
Hence for any $g \in \G_i(S)$ and for all $f \in J$ we have $f(Z_iX)(Y_ig)=0$.\\ We compute $f(Z_iX)(Y_ig)=f(Yg)$ and conclude that $g$ is contained in $\G(S)$. Therefore, $\G_i(S) \leq \G(S)$ holds for all $1 \leq i \leq r$. \\
\\
Let $\Hcal \leq \G$ be the smallest ${\del_t}$-closed subgroup containing $\G_i$ for all $1\leq i \leq r$. We claim that $\Hcal=\G$. Let ${R=F\{Y,Y^{-1}\}_{\del_t}}\subseteq {F\<Y\zu_{\del_t}=E}$ be the PPV-ring and the PPV-extension for $A$ over $F$. Similarly, let $R_{i}=F_{i}\{Y_{i},Y_{i}^{-1}\}_{\del_t}\subseteq F_i\<Y_i\zu_{\del_t}=E_i$ be the PPV-ring and the PPV-extension for $A_{i}$ over $F_i$, $1\leq i \leq r$. Consider an element $x \in E^\Hcal$, i.e., $x$ is functorially invariant under $\Hcal$. This means that we can write $x=\frac{a}{b}$ with $a,b \in R$ such that for all ${\del_t}$-$C$-algebras $S$ and for all $\sigma \in \Hcal(S)$ we have 
$\sigma(a\otimes_C 1)\cdot(b\otimes_C 1)=(a \otimes_C 1)\cdot \sigma(b\otimes_C 1)$. Note that $R=F\{Y,Y^{-1}\}_{\del_t}\subseteq F_{i}\{Y,Y^{-1}\}_{\del_t}=F_{i}\{Y_{i},Y_{i}^{-1}\}_{\del_t}=R_{i}$. Hence for all $1\leq i \leq r$, $x$ is contained in $E_{i}$ and is functorially invariant under $\G_{i}\leq \Hcal$. (Here we use that the embedding $\G_i\leq \G$ is compatible with the action of $\G$ on $E$ for all $1\leq i \leq r$. Indeed, the action of an element $g \in \G_i(S)$ is given by $Y_i\otimes 1 \mapsto (Y_i\otimes 1)\cdot g$ which translates to $Y\otimes 1 \mapsto (Y\otimes 1)\cdot g$ inside $\G(S)$, since $Y=Z_iY_i$.) It now follows from Part \ref{galois3}) of Theorem \ref{thmgalois} that $x$ is contained in $F_{i}$ for all $1\leq i \leq r$. Note that $E\subseteq F_{0}$ since $Y \in \GL_n(F_{0})$. Hence $x$ is also an element of $F_{0}$. Part a) of Theorem \ref{patchingarithmcurves} implies $x \in F$. Therefore, $E^\Hcal=F=E^\G$ and the Galois correspondence (see \cite[Proposition 8.5]{Gilletetc}) implies $\Hcal=\G$.
\end{proof}

\begin{rem}
Theorem \ref{thmpatching} can be generalized to fields with more than one parameter, as long as simultaneous factorization as in Theorem \ref{patchingarithmcurves} still holds. An example of such a field is $F=k((t))(x)$ with $k=k'(t_1,\dots, t_m)$ or any other parameterized field and with fields $F \subseteq F_i, F_0 \subseteq F_i^\circ$ defined as in the beginning of this section. Then Theorem \ref{thmpatching} also holds for the parameterized Galois groups with respect to $\Delta=\{\del_t,\del_{t_1},\dots,\del_{t_m}\}$.
\end{rem}

\section{Generating Kolchin-dense subgroups}\label{3}
In this section, we consider linear differential groups over $(k((t)),\del_t)$ with $k$ a fixed field of $\operatorname{char}(k)=0$. An imporant fact from differential algebra is that every differential field is contained in a differentially closed field $U$, which means that any set of $\del_t$-polynomial equations that has a solution in some $\del_t$-extension field of $U$ also has a solution inside $U$. We fix a differentially closed field $U$ containing $k((t))$. The following proposition generalizes the fact that a semisimple linear algebraic group is generated by its root subgroups.

\begin{prop}\label{propgen}
Let $\Hcal \leq \GL_n$ be a semisimple connected linear algebraic group that is defined over $k((t))$ and is $k((t))$-split. Fix a  a root system $\Phi$ of $\Hcal$ with $\Phi^+\subseteq \Phi$ a system of positive roots, $\mathcal{D} \subseteq \Phi^+$ a set of simple roots, $U_\alpha$ root subgroups defined over $k((t))$ and $u_\alpha \colon \mathbb{G}_a \to U_\alpha$ isomorphisms over $k((t))$. Then the following holds: Every $\del_t$-linear subgroup $\G \leq \Hcal$ defined over $k((t))$ (in other words, every Kolchin-closed subgroup defined over $k((t))$) with the following properties a) and b) equals $\Hcal$. 
\begin{enumerate}
 \item for all $\alpha \in \mathcal{D} $: $u_{\pm \alpha}(\pm 1) \in \G(k((t)))$
\item for all $\alpha \in \Phi^+$, there exists an $f_\alpha \in k((t))$ transcendent over $k$ with $u_\alpha(f_\alpha) \in \G(k((t)))$ and $u_{-\alpha}(-f_\alpha^{-1}) \in \G(k((t)))$
\end{enumerate}
\end{prop}
\begin{proof}
Assume first that $\Hcal$ is defined over $\Q$. Let further $\Hcal_1,\dots,\Hcal_r$ denote the quasisimple components of $\Hcal$. 
 Let $\G \leq \Hcal$ be a group as described in the theorem.
We set $\G_i=(\Hcal_i\cap \G)^\circ$ for $1 \leq i \leq r$. Note that $u_\alpha(1)$ generate Zariski-dense subgroups of $U_\alpha$, hence the Zariski closure of $\G$ contains $U_{\pm \alpha}$ for all $\alpha \in \mathcal{D}$ and thus equals $\Hcal$. Then by Theorem 15 in \cite{Cassidyclassification}, $\G_i$ is Zariski dense in $\Hcal_i$ and $\G$ is an almost direct product of $\G_1,\dots, \G_r$. Hence it suffices to show that $\G_i=\Hcal_i$ holds for all $i \leq r$. Let $i \leq r$. Theorem 19 in \cite{Cassidyclassification} implies that either $\Hcal_i=\G_i$ holds or that there exists a $g_i \in \Hcal_i(U)$ such that  
$\G_i(S)=g_i^{-1}\cdot\{h \in \Hcal_i(S) \ | \ \del_t(h)=0 \}\cdot g_i$ holds for all $\del_t$-$C$-algebras $S$ containing $\overline{ \Q}$. \\
\\
Assume that $\G_i\neq \Hcal_i$ holds.  Note that $\Hcal_i$ is semisimple with root system $\Phi_i=\{\alpha \in \Phi \ | \ U_\alpha \subseteq \Hcal_i\}\neq \varnothing$. Let $\alpha \in \Phi_i$. By assumption, $u_\alpha(f_\alpha)u_{-\alpha}(f_\alpha^{-1})u_{\alpha}(f_{\alpha}) \in (\G\cap \Hcal_i)(k((t)))$. Now by  \cite[Lemma 8.1.4]{Springer}, \[u_\alpha(f_\alpha)u_{-\alpha}(f_\alpha^{-1})u_\alpha(f_\alpha)=\alpha^\vee(f_\alpha)\cdot n_\alpha,\] with $\alpha^\vee$ a coroot corresponding to $\alpha$ and $n_\alpha=u_\alpha(1)u_{-\alpha}(-1)u_\alpha(1)$ also an element of $(\G\cap \Hcal_i)(k((t)))$. Hence $\alpha^\vee(f_\alpha) \in (\G\cap \Hcal_i)(k((t)))$, so there exists an $l \in \N_{\geq 1}$ with $\alpha^\vee(f_\alpha)^l=\alpha^\vee(f_\alpha^l) \in \G_i(k((t)))\subseteq g_i^{-1}\Hcal_i(\overline{k})g_i$ and we conclude that all eigenvalues of $\alpha^\vee(f_\alpha^l)$ are contained in $\overline{k}$ which implies $f_\alpha^l \in \overline{k}$, contradicting $f_\alpha \notin \overline{k}$.  Hence $\G_i=\Hcal_i$. 
\\
\\
The classification of reductive groups implies that every $k((t))$-split reductive group is isomorphic to one defined over $\Q$.
Now for a general $\Hcal$ that splits over $k((t))$, let $\gamma \colon \Hcal \to \tilde \Hcal$ be a $k((t))$-isomorphism with $\tilde \Hcal$ defined over $\Q$. The $k((t))$-split root systems $\tilde \Phi$ (i.e., all $\alpha$, $U_\alpha$ and $u_\alpha$ are defined over $k((t))$) of $\tilde \Hcal$ associated to a $k((t))$-split maximal torus $\tilde T\leq \tilde \Hcal$ are in one-to-one correspondence with the $k((t))$-split root systems $\Phi$ of $\Hcal$ associated to the maximal torus $\gamma^{-1}(\tilde T)$ (via $\Phi=\{\tilde \alpha\circ\gamma \ | \ \tilde \alpha \in \tilde \Phi\}$, $U_{\tilde \alpha\circ\gamma}=\gamma^{-1}(U_{\tilde \alpha})$ and $u_{\tilde \alpha\circ\gamma}=\gamma^{-1}\circ u_{\tilde \alpha}$ ). Let $\G \leq \Hcal$ be a Kolchin-closed subgroup defined over $k((t))$ such that a) and b) holds. Then $\gamma(\G) \leq \tilde \Hcal$ satisfies a) and b) with respect to the root system $\tilde \Phi$, hence $\gamma(\G)=\tilde \Hcal$ by what we proved above and we conclude $\G=\Hcal$.

\end{proof}

\section{Semisimple linear algebraic groups as parameterized differential Galois groups}\label{4}
Let again $k$ be a field of $\operatorname{char}(k)=0$ and consider $k((t))(x)$ equipped with the natural derivations $\del=\frac{\del}{\del x}$ and $\del_t=\frac{\del}{\del t}$.
\begin{lem}\label{lemmaschnitt}
Let $q \in k$. Then $$\operatorname{Frac}(k[[t]][[x-q]])\cap \operatorname{Frac}(k[(x-q)^{-1}][[t]])=k((t))(x),$$ where the intersection is considered inside $k((x-q))((t))$.
\end{lem}
\begin{proof}
This follows directly from Theorem \ref{patchingarithmcurves}.a) with $r=1$.
\end{proof}

\begin{thm}\label{result}
Let $\Hcal \leq \GL_n$ be a connected semisimple linear algebraic group defined and split over $k((t))$. Then $\Hcal$ is the parameterized Galois group of some $n$-dimensional $\del$-differential equation over $k((t))(x)$.
\end{thm}
\begin{proof}
Let $\Phi$ be a set of roots of $\Hcal$ defined over $k((t))$ and fix a set of positive roots ${\Phi^+=\{\alpha_1,\dots,\alpha_m\}} \subseteq \Phi$. We also fix root subgroups $U_\alpha \leq \Hcal$ and $k((t))$-isomorphisms $u_\alpha \colon \mathbb{G}_a \to U_\alpha$ for each $\alpha \in \Phi=-\Phi^+\cup\Phi^+$. Fix pairwise distinct elements $q_1,\dots,q_{4m} \in k$. We set $F=k((t))(x)$ and $r=4m$. Consider the $\del\del_t$-fields $F\subseteq F_0, F_i \subseteq F_i^{\circ}$ ($1\leq i \leq r$) as defined in Section \ref{secpat}. We further let $C=k((t))$ be the field of $\del$-constants of these fields.\\
\\
We will construct $A_i \in F_i^{n\times n}$ for $1\leq i \leq r$ such that the $\del$-equations $\del(y)=A_iy$ over $F_i$  have fundamental solution matrices $Y_i \in \GL_n(F_0)\subseteq \GL_n(F_i^\circ)$ and parameterized Galois groups $\G_i \leq \GL_n$ satisfying the following. For all $1 \leq i \leq m$ and all $\del_t$-$C$-algebras $S$ the following holds:
\begin{eqnarray}
\G_i(S)&=&U_{\alpha_i}(S^{\del_t}) \label{arraypatches1} \\
\G_{m+i}(S)&=&U_{-\alpha_i}(S^{\del_t})\\
\G_{2m+i}(S)&=&U_{\alpha_i}(S^{\del_t}\cdot t) \\
\G_{3m+i}(S)&=&U_{-\alpha_i}(S^{\del_t}\cdot t^{-1}) \label{arraypatches4}
\end{eqnarray}
In other words, $\G_i \leq U_{\alpha_i}$ and $\G_{m+i} \leq U_{-\alpha_i}$ are the constant subgroup schemes and $\G_{2m+i}=u_{\alpha_i}(G_i)$ with $G_i \leq \mathbb{G}_a$ the $\del_t$-subgroup given by the equation $\del_t(X)-\frac{1}{t}X=0$ and similarly $\G_{3m+i}=u_{-\alpha_i}(\tilde G_i)$ with $ \tilde G_i \leq \mathbb{G}_a$ the $\del_t$-subgroup given by the equation $\del_t(X)+\frac{1}{t}X=0$.\\
Theorem \ref{thmpatching} will then imply that there is an $A \in F^{n\times n}$ such that the parameterized Galois group $\G \leq \GL_n$ of the $\del$-equation $\del(y)=Ay$ over $F$ is the smallest $\del_t$-subgroup of $\GL_n$ containing $\G_i$ for $1\leq i \leq 4m$. Then $\G=\Hcal$ holds by Proposition \ref{propgen} (applied to $f_\alpha=t$ for all $\alpha$).\\
\\
For $1\leq i \leq 4m$, set $f_i:=\sum\limits_{n=1}^{\infty}\frac{(-1)^{n+1}}{n(x-q_i)^n}t^n \in F_0$. Then $$\del(f_i)=\frac{1}{x-q_i}\sum\limits_{n=1}^{\infty}\left(\frac{-t}{x-q_i}\right)^n=\frac{1}{x-q_i+t}-\frac{1}{x-q_i}=:a_i \in F$$
and $$\del_t(f_i)=\frac{-1}{t}\sum\limits_{n=1}^{\infty}\left(\frac{-t}{x-q_i}\right)^n=\frac{1}{x-q_i+t} \ \in F.$$
For $1 \leq i \leq 4m$, we define $E_i=F_i\<f_i\zu_{\del_t}=F_i(f_i)$. Then $E_i$ is a PPV-extension for the $\del$-equation over $F_i$ given by $\tilde A_i=\begin{pmatrix}
0& a_i\\
0 &0                                                                                                                               \end{pmatrix}$ with fundamental solution matrix $\tilde Y_{i}=\begin{pmatrix}
1& f_i\\
0 &1                                                                                                                               \end{pmatrix} \in \GL_2(F_0)$ and PPV-ring $R_i=F_i\{\tilde Y_{i}, \tilde Y_{i}\}_{\del_t}=F_i\{f_i\}_{\del_t}=F_i[f_i]$.
Let $S$ be a $\del_t$-$C$-algebra. Then an $(F\otimes_C S)$-automorphism $\sigma$ of $R_i\otimes_C S=(F_i\otimes_C S)[f_i\otimes 1]$ is given by $\sigma(f_i\otimes 1)$ and $\sigma$ commutes with $\del$ and $\del_t$ if and only if $\sigma(f_i\otimes 1)=f_i\otimes 1+a_\sigma$ with $\del(a_\sigma)=\del_t(a_\sigma)=0$ which is the case if and only if $a_\sigma \in ((R\otimes_C S)^{\del})^{\del_t}=S^{\del_t}$. Hence $\Gal(\tilde A_i)$ is a subgroup scheme of the constant subscheme of $\mathbb{G}_a$. As $f_i \notin F$, Lemma \ref{lemmaschnitt} implies that $f_i \notin F_i$, so $\Gal(\tilde A_i)$ is not trivial. Hence $\Gal(\tilde A_i)$ equals the constant subscheme of $\mathbb{G}_a$, i.e., for every $a \in S^{\del_t}$ there exists a (unique) $\sigma_a \in \Aut^{\del_t\del}(R_i\otimes_C S/F\otimes_C S)$ with $\sigma_a(f_i\otimes 1)=f_i\otimes 1+a$. We conclude $\Aut^{\del_t\del}(R_i\otimes_C S/F\otimes_C S)=\{ \sigma_a \ | \ a \in S^{\del_t} \}$. \\
\\
For $1 \leq i \leq m$, we now set 
\begin{eqnarray*}
Y_i&=&u_{\alpha_i}(f_i) \in \GL_n(R_i) \\
Y_{m+i}&=&u_{-\alpha_i}(f_{m+i}) \in \GL_n(R_{m+i}) \\
Y_{2m+i}&=&u_{\alpha_i}(t\cdot f_{2m+i}) \in \GL_n(R_{2m+i})\\
Y_{3m+i}&=&u_{-\alpha_i}(t^{-1} \cdot f_{3m+i}) \in \GL_n(R_{3m+i}) 
\end{eqnarray*}
and we set $A_i=\del(Y_i)Y_i^{-1} \in R_i^{n\times n}$ for $1\leq i \leq 4m$. We use that all $u_{\alpha_i}$'s are defined over $k((t))$ to see that the entries of $A_i$ are functorially invariant under $\Gal(\tilde A_i)$, so we actually have $A_i \in F_i^{n\times n}$ (see Part \ref{galois3}) of Theorem \ref{thmgalois}). As $u_{\alpha_i}$ is a $k((t))$-isomorphism, we have $F_i\<Y_i, Y_i^{-1} \zu_{\del_t}=F_i\<f_i\zu_{\del_t}=E_i$. Hence $E_i$ is a PPV-extension for $A_i$ and $R_i$ is a PPV-ring for $A_i$ ($1\leq i \leq 4m$). Let $\G_i \leq \GL_n$ be the image of $\Gal(A_i)$ with respect to the embedding corresponding to $Y_i$ (see Theorem \ref{thmgalois}\,\ref{galois1}). We claim that 
(\ref{arraypatches1})-(\ref{arraypatches4}) holds. Let $1\leq i \leq 4m$, $S$ a $\del_t$-$C$-algebra and $a \in S^{\del_t}$. Now $u_{\alpha_i} \colon \mathbb{G}_a\to U_{\alpha_i}$ is defined over $k((t))=C$, hence it commutes with $\sigma_a$. It follows that 
\[Y_i^{-1}\sigma_a(Y_i)=\begin{cases}
                         u_{\alpha_i}(a) &\mbox{if } 1 \leq i \leq m \\
u_{-\alpha_i}(a) &\mbox{if } m+1 \leq i \leq 2m \\ 
u_{\alpha_i}(t\cdot a) &\mbox{if } 2m+1 \leq i \leq 3m \\ 
u_{-\alpha_i}(t^{-1} \cdot a) &\mbox{if } 3m+1 \leq i \leq 4m \\  
                        \end{cases}
 \] holds and the claim follows.
\end{proof}

\begin{rem}
We expect that patching can also be used to show that other classes of linear differential groups occur as parameterized Galois groups over $k((t))(x)$. There is work in progress to study whether every linear differential group $\G$ defined over $k((t))$ that is differentially finitely generated over $k((t))$ (by which we mean that it can be generated by finitely many $k((t))$-rational elements $g_1,\dots,g_m$) is a parameterized Galois group over $k((t))(x)$. This question seems suitable for an application of patching, since $\G$ is then differentially generated by the differential closures $\G_i$ of $<\!g_i\!>$ for $i \leq m$ and these groups $\G_i$ can be described explicitly. 
\end{rem}

\bibliographystyle{alpha}	
 \bibliography{references}

\end{document}